\documentclass{amsart}
\usepackage{amscd, amsfonts, amsmath, amsthm, amssymb,epsf}

\usepackage{tabularx}
\usepackage{ltablex}
\usepackage{longtable}
\usepackage{multirow}
\usepackage{hhline}
\usepackage{fancyhdr}

\usepackage{stmaryrd}
\usepackage[all]{xy}
\usepackage{srcltx}

\newcommand{\Fq}{\mbox{${\mathbb F}_{p}$}}
\newcommand{\cal}{\mathcal}
\newcommand{\Nr}{{\cal N}}
\newcommand{\qp}{Q_{\wp}}

\newtheorem{theorem}{Theorem}

\newcommand{\Z}{\mathbb{Z}}
\newcommand{\F}{\mathbb{F}}
\newcommand{\Q}{\mathbb{Q}}

\DeclareMathOperator{\Norm}{Norm}

\begin{document}

\author{Dinesh S. Thakur}

\title[Fermat-Wilson supercongruences]{Fermat-Wilson Supercongruences,   arithmetic derivatives and strange factorizations}

\address{ Department of Mathematics, University of Rochester, 
Rochester, NY 14627, dinesh.thakur@rochester.edu
}

\date{\today}

\begin{abstract}
In \cite{fw}, we looked at two (`multiplicative' and  `Carlitz-Drinfeld additive') analogs each, for  the well-known basic congruences of Fermat and Wilson,  
in the case of polynomials over finite fields. 
When we look at them modulo 
higher powers of primes, i.e. at `supercongruences', we  find interesting relations 
linking them together,
as well as  linking them with  arithmetic derivatives and zeta values. 
In the current work, we expand on the first analog and connections with arithmetic derivatives more systematically, giving many more equivalent conditions linking the two, 
now using `mixed derivatives' also. We also observe and prove  remarkable prime factorizations involving derivative conditions for some fundamental quantities of the function field arithmetic.
\end{abstract}


\maketitle

\section{Introduction}

In the well-known number field-function field analogy, the cyclic group /multiplicative group of finite fields nature of  $(\Z/p\Z)^*$ and $(\F_q[t]/\wp\F_q[t])^*$ , which are (where $p$ is a prime in $\Z$ and $\wp$ is an irreducible polynomial in $\F_q[t]$)
gives parallel statements and proofs for  Fermat's little theorem and the Wilson theorem. But the different group theoretic nature of $(\Z/p^k\Z)^*$ (cyclic for odd prime $p$) from  $(\F_q[t]/\wp^k \F_q[t])^*$, for higher $k$ is partly responsible for the fact that the question of infinitude of Wieferich or Wilson primes is mysterious and still open for integers, while 
at least a naive (multiplicative) analog has a nice complete answer \cite{Tw, fw} for the polynomial case. Somewhat surprisingly, it involves derivatives very intimately. 

In this paper, we explore this further and prove more results by providing equivalence between supercongruences and  vanishing of various arithmetic higher (pure (Theorem 1 and 3) or mixed (Theorem 4)) derivatives. 

Next we show (Theorem 7) that while the fundamental `numbers'  $D_d, L_d$ of the function field arithmetic (see the start of Section 7 for definitions, factorizations) have beautiful regular symmetric prime  factorizations exactly involving all the primes  of degree $\leq d$, with multiplicities simply depending on just their degrees; simple perturbations   $D_{d-1}-c$ or $L_{d-1}-c$ ($c\in \F_q^*$) have `derivative constant' characterizations of the  degree $d$ primes occurring in their factorizations, leading to  very few `special Wilson primes' as factors, and prime factors of larger degrees being mysterious.  We also note that $D_d$ and $L_d$ occur as (reciprocal) coefficients in Carlitz-Drinfeld exponential and logarithm series for $\F_q[t]$ respectively, and that $D_d$ is the Carlitz factorial of $q^d$. For more,  we refer to \cite[Sec. 2.5, 4.13]{Thakur}.

Interestingly, the Wilson primes (`double derivative 0' condition) are exactly the primes involved in a strange hybrid version (Theorem 5) of the famous Wolstenholme theorem that $p^2$ divides $1+1/2+\cdots +1/(p-1)$, for prime $p>3$.  See 7.4 for some numerical examples of Theorems 7 and 5 for more on these strangely beautiful factorizations.

\section{Basic definitions and Fermat-Wilson analogs} 

Let $A=\F_q[t]$, where $\F_q$ is a finite field of $q$ elements, where $q$ is a power of a prime $p$.  Let $\wp$ denote a monic prime of $A$ of degree $d$ (in $t$), so that its residue field 
$\F_{\wp}\subset A_{\wp}$ has cardinality $\Norm \wp =q^d$. Let $\theta\in \F_{\wp}$ be the Teichm\"uller representative 
of $t$ modulo $\wp$. Note that $\wp=\prod (t-\theta^{q^i})$ is the minimal polynomial in $A$ for $\theta$.
`
We have the well-known Fermat theorem analog: $a^{\Norm \wp}\equiv a\mod \wp$, for $a\in A$.

\subsection{Definition} Let $a\in A$. We say that $\wp$ is a Wieferich prime base $a$ (or $a$-Wieferich), if $a^{\Norm \wp}\equiv a \mod \wp^2$.

Often, but not here,  one excludes in the definition the trivial cases $a=0, 1, -1$ classically,  and $a\in \F_q$ in the function field case, as these are exactly the cases 
where $a^p=a$ and $a^{\Norm \wp} =a$ respectively. 

We have the well-known Wilson theorem analog:  $F_d\equiv -1\mod \wp$, where $F_d$ is the product of all non-zero polynomials of degree less than $d$
(which represent `smallest' representatives of all non-zero residue classes modulo $\wp$). (For more analogies, using the Carlitz factorial,  we refer to \cite{fw}.)

\subsection{Definition} We say that $\wp$ is a Wilson prime, if $F_d\equiv -1\mod \wp^2$.

\section{Three arithmetic derivatives}

We now give definitions and some basic comments on the three arithmetic derivatives, in fact, a derivative, a Frobenius-difference quotient and a difference quotient. Let $a\in A\F_{\wp}=\F_{\wp}[t]$. 

 \subsection{The usual derivative} Let $a^{(1)}:=D(a): =da/dt$ and  denote by $a^{(i)}:=D^i(a):=d^ia/dt^i$.

 \subsection{Fermat quotient (i.e., Frobenius-difference quotient) derivative} Let $\qp(a):= (a^{\Norm \wp} - a)/\wp$ and  denote its $i$-th iteration by $\qp^{i}$.

 \subsection{Teichm\"uller difference quotient derivative}   Define $a^{[i]}=\Delta^i(a)$  by $a^{[0]}(t)=a(t)$ and $a^{[i+1]}(t)=(a^{[i]}(t)-a^{[i]}(\theta))/(t-\theta)$.

\subsection{Remarks}  (I) All these derivatives give self-maps on $A\F_{\wp}$, and the first two restrict to self-maps on $A$ also. They are all $\F_{\wp}$-linear. 
They depend on the choice $t$ of the generator of $A$ only through its sign. 

(II) They all evaluate to zero on constants $a\in \F_{\wp}$. Evaluated on $p$-th powers, the first one vanishes, the second one vanishes modulo $\wp$, and the 
second and third one vanish when evaluated at $t=\theta$. 

(III) Let us denote the degree in $t$ by $\deg$. When $\deg(a)>0$, (i) $\deg (da/dt)\leq \deg (a) -1$, with strict inequality, exactly when $p$ divides $\deg(a)$, (ii) $\deg(\qp(a))
=q^d\deg(a)-d, (iii) \deg a^{[1]}=\deg(a)-1$. 

(IV) For $f=\sum f_it^i\in F_{\wp}[t]$, we have (i) $da/dt=\sum if_it^{i-1}$,   (ii) $\qp(a)=\sum f_i (t^{iq^d}-t^i)/\wp$, (iii) $a^{[1]}=\sum f_i\sum_{j=1}^i t^{i-j}\theta^{j-1}$. 

(V) Part (iii) of (IV) implies that $a^{[1]}|_{t=\theta}=da/dt|_{t=\theta}$ (formally, without using definition of $\theta$) and similarly for higher derivatives. This also follows from the 
fact that these higher differences are polynomials which are continuous, so the derivative-limit is the evaluation.


(VI) We will not give corresponding (twisted) derivatives properties for each, as we do not need them. But see e.g., \cite{buium} for analogous set-ups in characteristic zero.

\section{Fermat supercongruence and the first derivative of the base}

By the definition, the condition that `$\wp$ is $a$-Wieferich' is equivalent to `$Q_{\wp}(a)\equiv 0\mod \wp$'. The condition being equivalent to higher multiplicity of the root $\theta$, by the usual 
detection of such multiplicities by derivatives in the polynomial case, we get  more transparent (e.g, (i) below) equivalent vanishing derivative conditions, as follows.

\begin{theorem} The following conditions are equivalent. 

(0) Prime $\wp$ is $a$-Wieferich, 

(i) $da/dt \equiv 0 \mod \wp$, (i') $(da/dt)|_{t=\theta}=0$,

(ii) $Q_{\wp}(a)\equiv 0\mod \wp$, (ii') $Q_{\wp}(a)|_{t=\theta}=0$, 

(iii) $a^{[1]}|_{t=\theta}=0$. 

\end{theorem}

\begin{proof}
The equivalence of (0) and (ii) follows from definitions.  The equivalence with (i) was also noted e.g., in \cite[p.195]{fw}. The equivalence of (ii) with (ii'), and of (i) with (i') follow,  since $\wp$ is minimal polynomial over $A$ of $\theta$. The equivalence (i') with (iii) was noted in Remarks 3.4 (V). 
\end{proof}

As an immediate corollary, we get 

\begin{theorem}
(i) If there are infinitely many $a$-Wieferich primes, then $a=b^p$ for some $b\in A$, and then all the primes of $A$ are $a$-Wieferich. 

(ii) There are no $a$-Wieferich primes, if and only if $a=b^p+ct$, with $b\in A$ and $c\in \F_q^*$. 

\end{theorem}

\subsection{Remarks} We can also see this \cite[p. 195]{fw} from the following. 
 If $a=\sum a_it^i$,
  we have, modulo $\wp^2$, (without loss of generality $\wp\neq t$) that
$$a^{q^d}-a = \sum a_it^i((t^{q^d-1}-1+1)^i-1)\equiv \sum a_it^i
{i\choose 1}([d]/t)^1 = (da/dt)[d],$$

This suggests that while we do not have distinguished `$t$' in the rational numbers case to compare $da/dt$, given $p$, a mod $p$ analog of $db/da$ may be the ratio of 
Fermat quotients $(b^p-b)/(a^p-a)$. 

If we take $a=\wp$ and divide the displayed congruence  by $\wp$ we see that 
$$d\wp/dt\equiv -1/Q_{\wp}(t) \mod \wp.$$ Since $d\wp/dt$ has degree less than that of $\wp$, 
this allows us to extract $d\wp/dt$ from $Q_{\wp}(t)$ modulo $\wp$.  In contrast to $d\wp/dt$ and $\wp^{[1]}$, the Fermat quotient $Q_{\wp}(t)$, which  occurs analogously in our main theorems, feels  more like the derivative of  $t$ rather than of $\wp$ with respect to $t$. This explains the reciprocal relation.

\section{Wilson supercongruence and the second derivatives of the prime}

We restrict to $p>2$, for simplicity, leaving the $p=2$ discussion to \cite{fw}. 

.

\begin{theorem} Let $p>2$. The following are equivalent. 

(0) Prime $\wp$ is a Wilson prime, 

(i) $d^2\wp/dt^2=0$, (i') $d^2\wp/dt^2\equiv 0 \mod \wp$, (i'') $(d^2\wp/dt^2)|_{t=\theta} =0$,

(ii) $\qp^2(t) \equiv 0 \mod \wp$, (ii') $\qp^2(t)|_{t=\theta}=0$, 

(iii) $\wp^{[2]}|_{t=\theta}=0$. 

\end{theorem}

\begin{proof} The equivalence of (0), (i), (ii) and (iii) was proved in the main theorems \cite[Thm. 2.5, Thm. 2.9]{Tw}. The equivalence of (i) with (i') and (ii) with 
(ii') follows as before. 
\end{proof}

\subsection{Remarks}: As a corollary, we got a  simple characterization of Wilson primes $\wp=\sum p_it^i$ as irreducible polynomials with  (from (i)) non-zero $p_i$ occurring only when $p$ divides $i$ or $i-1$. We 
deduced  \cite[Thm. 2.10]{Tw} their  infinitude for any given $A$. It was also proved \cite[Thm. 2.9]{Tw} that if  the Wilson congruence holds modulo $\wp^2$, it automatically holds modulo $\wp^{p-1}$.

To these three `pure' double derivatives conditions, we now add six more `mixed'  double derivatives equivalent conditions. 

\begin{theorem}
The following are equivalent to the conditions of the previous theorems. 

(i-ii) $dQ_{\wp}(t)/dt \equiv 0 \mod \wp$, (i-ii') $(dQ_{\wp}(t)/dt)|_{t=\theta}=0$, 

(ii-i) $Q_{\wp}(d\wp/dt)\equiv 0 \mod \wp$, (ii-i') $ (Q_{\wp}(d\wp/dt))|_{t=\theta}=0$, 

(i-iii) $(d\wp^{[1]}/dt)|_{t=\theta}=0$, 

(iii-i) $(d\wp/dt)^{[1]})|_{t=\theta}=0$, 

(ii-iii) $(Q_{\wp}(\wp^{[1]}))|_{t=\theta}=0$, 

(iii-ii) $(Q_{\wp}(t)^{[1]})|_{t=\theta}=0$.

\end{theorem}

\begin{proof} The equivalence of (i-ii) with the `primed' version (i-ii') and of (ii-i) with (ii-i') follows since $\wp$ is the minimal polynomial over $A$ for $\theta$. 
We now use freely the remarks 3.4 in the proof, and show the equivalences, one by one, to some previously established ones. 

The equivalence of (i-ii)  (\cite[Thm. 2.6]{fw})  follows from the first congruence in Remarks 4.1 specialized at  $a=Q_{\wp}(t)=[d]/\wp$, 
where $[d]=t^{q^d}-t$, since $\wp$ divides $[d]$ with multiplicity $1$.  

(i) implies that $d\wp/dt=a^p$ for some $a\in A$, so that its Fermat quotient is divisible by $\wp^{p-1}$, so it implies (ii-i). Conversely, (ii-i) implies that 
$(\wp')^{q^d}-\wp'=f\wp^2$, for some $f\in A$. Taking derivative with respect to $t$, we get $f'\wp^2+2f\wp\wp'=-\wp''$, so that $\wp$ divides $\wp''$, which by the degree considerations immediately implies (i). 

By the quotient rule, (i-iii) implies that $(t-\theta)^3$ divides $\wp'(t)(t-\theta)-(\wp(t)-\wp(\theta))$, so that $(t-\theta)^2$ divides $\wp''(t)(t-\theta)+\wp'(t)-\wp'(t)=\wp''(t)(t-\theta)$, which implies (i''). Conversely, since $\wp'(t)=\wp'(\theta)$ by  Remark 3.4 (V), we have 
$$\frac{d}{dt}\wp^{[1]}(t)=\frac{\wp'(t)}{t-\theta}-\frac{\wp(t)-\wp(\theta)}{(t-\theta)^2}=\frac{\wp'(t)-\wp'(\theta)}{t-\theta}-\frac{\wp^{[1]}(t)-\wp^{[1]}(\theta)}{t-\theta}.$$
Now the second quantity vanishes at $t=\theta$ by (iii),  and by (ii) the numerator of the first quantity is divisible by $(t-\theta)^p$. This implies (i-iii). 

(iii-i) implies that $(t-\theta)^2$ divides $\wp'(t)-\wp'(\theta)$, so that $t-\theta$ divides $\wp''(t)$ implying (i''). Conversely, (i) implies $\wp'$ is $p$-th power, so that 
$(\wp')^{[1]}$ is divisible by $(t-\theta)^{p-1}$ implying (iii-i). 

(ii-iii) implies $(t-\theta)\wp$, and so also $(t-\theta)^2$, divides $(\wp^{[1]})^{q^d}-\wp^{[1]}$. Taking the derivative with respect to $t$, we see that $t-\theta$ divides 
$d/dt(\wp^{[1]})$ implying (i-iii). Conversely, (iii) implies that $(t-\theta)^2$ divides $:=\wp^{[1]}(t)-\wp^{[1]}(\theta)$, so it divides $x^{q^d}-x=(\wp^{[1]}(t))^{q^d}-\wp^{[1]}(t)$ implying (ii-iii). 

(iii-ii) is equivalent to the divisibility of $Q_{\wp}(t)-Q_{\wp}(t)|_{t=\theta}$, which is equivalent to the divisibility if $d/dt(Q_{\wp}(t))$ by $t-\theta$, which is (i-ii'). 
\end{proof} 

\subsection{Remarks} The condition (ii-i) can be restated in a more striking form saying  that $\wp$ is Wilson if and only if $\wp$  is base $d\wp/dt$-Wieferich.



\section{Examples} 

Let us  verify, by direct calculations, the nine equivalent conditions of the two main theorems for the family of Artin-Schreier primes $\wp= t^p-t-m$ for $A=\F_p[t]$, m$\in \F_p^*$, and $p>2$. 
That these primes are Wilson primes was  noted and proved already in \cite{binom}[Thm. 7.1]. 

Since $D\wp=-1, D^2\wp=0$, we get (i). 

Since $\wp^{[1]} =(t-\theta)^{p-1}-1$, $\wp^{[2]}=(t-\theta)^{p-2}$, we get (iii). 

The calculation  $\qp(t)=(t^{p^p}-t)/(t^p-t-m)= \wp^{p^{p-1}}+\wp^{p^{p-2}}+\cdots + \wp^{p-1}+1$ shows that $\wp^{p-2}$ divides $Q_{\wp}^2(t)$ verifying (ii). 

This calculation also implies (i-ii) immediately. 

We have $Q_{\wp}(\wp'(t))=Q_{\wp}(-1)=0$ implying (ii-i). 

We have $d/dt(\wp^{[1]})=(p-1)(t-\theta)^{p-2}$ implying (i-iii). 

Since $\wp'=-1$, we have $(\wp')^{[1]}=0$ implying (iii-i). 

We see that $Q_{\wp}(\wp^{[1]})=Q_{\wp}((t-\theta)^{p-1}-1)=[((t-\theta)^{p-1}-1)^{p^p}-((t-\theta)^{p-1}-1)]/\wp$ is divisible by $(t-\theta)^{p-2}$, hence  we have (ii-iii). 

Since $\wp(\theta)=0$, the calculation above of $Q_{\wp}(t)$ shows that $(Q_{\wp}(t))^{[1]}=(Q_{\wp}(t)-1)/(t-\theta)$ is divisible by $(t-\theta)^{p-2}$ and we verify (iii-ii).

\section{Derivative conditions on primes occurring in some natural factorizations}

Let us recall some basic quantities/notation  from Carlitz associated to the arithmetic of $A$. For a non-negative integer $n$, we put $[n]=t^{q^n}-t$.
We put $L_0=D_0=1$ and for a positive integer $n$, we put $L_n=[n]L_{n-1}, D_n=[n]D_{n-1}^q$. 

Recall (see e.g., \cite{Thakur}[Sec. 2.5]) the nice factorizations of these fundamental quantities: The quantity $[d]$ is the product of all (monic) primes of degree dividing $d$, $D_d$ is the product of all monic polynomials of degree $d$, and $L_d$ is the (monic) least common multiple of all polynomials of degree $d$. So $L_d=\prod \wp^{\lfloor d/k\rfloor}$, 
and $D_d =\prod \wp^{n_k}$, where both the products run over (monic) primes $\wp$ of degree $k\leq d$ and $n_k=\sum q^{d-ek}$: the sum over $1\leq e\leq \lfloor d/k\rfloor$.

\subsection{Second derivative condition: Wilson primes}

\begin{theorem} Let $p>2$. The degree $d$ primes dividing (the numerator of ) $1/[1]+1/[2]+ \cdots + 1/[d-1]\in  \F_q(t)$ (or equivalently, dividing the polynomial 
$-L_{d-1}' = L_{d-1}/[1]+\cdots +L_{d-1}/[d-1]$) are exactly the Wilson primes, i.e., the primes of degree $d$ with the vanishing second derivative.
These exist only if $p$ divides $d$ or $d-1$. They occur with multiplicity (at least) $p-2$. 
\end{theorem}

\begin{proof}
The Wilson primes of degree $d$ are, by definition, those $\wp$ which occur with multiplicity at least 2 in the factorization $F_d+1$, and we have 
\cite[p. 1842]{Tw}$F_d=(-1)^dD_d/L_d$. Hence exactly these $\wp$'s  divide the derivative 
$$\frac{D_d'L_d-D_dL_d'}{L_d^2}=\frac{-D_{d-1}^qL_d-[d]D_{d-1}^qL_d'}{L_d^2}=D_{d-1}^q\frac{L_{d-1}+L_d'}{[d]L_{d-1}^2}.$$
Now, by the product rule of derivatives, we have 
\begin{align*}
L_d'+L_{d-1} & =  -[d]([d-2]\cdots [1]+[d-1][d-3]\cdots [1]+\cdots)\\
 & =  -[d]L_{d-1}(\frac{1}{[d-1]}+\frac{1}{[d-2]}+\cdots +\frac{1}{[1]}),
 \end{align*}
so that the derivative above is $D_{d-1}^q/L_{d-1}$ times the first expression in the theorem. The first claim follows,  since by the above factorization 
results, $D_{d-1}$ or $L_{d-1}$ factorization does not involve any prime of degree $d$. The second claim follows from the characterization \cite[Thm. 2.9]{Tw} of Wilson primes, which implies that $d$ or $d-1$ has to be divisible by $p$. The final claim follows from the result \cite[Thm. 2.9]{Tw} that for Wilson primes $\wp$, the Wilson supercongruence hold mod $\wp^{p-1}$. 
\end{proof}

In \cite[Thm.7.1]{binom}, we showed that in addition to Wilson congruence, there is also `naive Wilson congruence' $[1]\cdots [p-1]\equiv -1 \mod \wp^{p-1}$, 
where $\wp$ is Artin-Schreier prime $\wp =t^p-t-c$ of $\F_q[t]$, where $q=p$. We had conjectured \cite[Pa. 281]{binom} with  some evidence,  that the existence of 
non-trivial gcd between $L_{d-1}+1=[1]\cdots [d-1]+1$ and $[d]$ implies $p$ divides $d$. This was proved (communication with the author, 15 October 2015) by Alexander Borisov.
The statement and the proof immediately generalizes to 

\begin{theorem} Let $d>1$. If $L_{d-1}+c$ and $[d]$ have a non-trivial gcd, where $c\in \F_q^*$, then $p$ divides $d$. 

\end{theorem}

\begin{proof} (Borisov) Non-trivial gcd implies  existence of a root $w\in \F_{q^d}$ for $L_{d-1}+c$. Put $x_j=w^{q^j}$. Then the root means 
$\prod_{j=1}^{d-1}(x_0-x_j)=(-1)^{d-1}(-c)=(-1)^dc$, and raising to $q^i$ powers gives $\prod_{j\neq i}(x_i-x_j)=(-1)^dc$. Now Lagrange interpolation at the $d$ points 
of the constant polynomial in $x$ gives $c=\sum_{i=0}^{d-1}c\prod_{j\neq i}(x-x_j)/\prod(x_i-x_j)$. Comparing coefficients of $x^{d-1}$ gives 
$0=\sum c/((-1)^sc)=(-1)^s\sum_{i=0}^{d-1} 1=d(-1)^s$, thus $p$ divides $d$. 
\end{proof}

Our conjecture \cite[2.2.5(i)]{fw} in connection  with `additive Wieferich-Wilson primes'  saying that `if $p>2$ and gcd between $[d]$ and $1-[d-1]+[d-1][d-2]-\cdots 
+(-1)^{d-1} L_{d-1}$ is non-trivial, then $p$ divides $d$' is still open.  

\subsection{First derivative condition: Special Wilson primes} 

Recall the nice regular factorizations of fundamental quantities $[d], L_d, D_d$ given above, where for a given degree, all the primes of that degree occur with the same 
(non-negative) multiplicity.  In contrast, we have: 

\begin{theorem} Let $d>1$.  For $c\in \F_q^*$, the degree $d$ primes dividing $L_{d-1}-c$ (are also  those dividing $D_{d-1}+(-1)^dc$) are exactly the degree $d$ primes $\wp$ with $d\wp/dt=(-1)^{d-1}c$, i.e., 
the monic primes $\wp=a^p+(-1)^{d-1}ct$ for some $a\in A$. 
These thus exist only when $p$ divides $d$. When they exist they occur with multiplicity at least $p-1$ (which seems to be even exact in `small  degree and $q$' data, 
except when $p=2$, $d=3$, when it is $2$.) for the $L$-case,  and with multiplicity  one for the $D$ case.

\end{theorem}

\begin{proof}
First note  the following simple calculation (in fact, equivalent to the Wilson congruence, since $F_d=(-1)^dD_{d-1}^q/L_{d-1}$) modulo $[d]$: 
$$D_{d-1}^q= \prod_{i=0}^{d-2}(t^{q^{d-1}}-t^{q^i})^q= \prod (t^{q^d}-t^{q^{i+1}})\equiv \prod (t-t^{q^{i+1}})= (-1)^{d-1}L_{d-1}.$$

Next note that, for $\wp$ a prime of degree $d$, we have  $([d]/\wp)D_{d-1}^q=D_d/\wp$, which is the product of all monic polynomials of degree $d$ not divisible by $\wp$, so modulo $\wp$ it is just $F_d\equiv -1$ by the Wilson congruence. In fact, $D_d/\wp\equiv F_d \mod \wp^{q-1}$, though we will not need this. This is seen by the mod $\wp^{q-1}$ calculation
$$\frac{D_d}{\wp}=\prod_{a\in A_+, \deg a <d}\prod_{\theta\in \F_q^*} (\wp +\theta a) =\prod (\wp^{q-1}-a^{q-1})\equiv \prod (-a^{q-1})=(-1)^dD_d/L_d=F_d.$$

Combining these two observations with the connection between the Fermat quotient and the derivative observed in Remarks 4.1, we see that  modulo a prime $\wp$ of degree $d$, we have 
$$L_{d-1}\equiv c \leftrightarrow D_{d-1}^q\equiv (-1)^{d-1}c\leftrightarrow Q_{\wp}(t)=[d]/\wp\equiv (-1)^d/c\leftrightarrow d\wp/dt\equiv (-1)^{d-1}c.$$

Since $D_{d-1}+c=[d-1]D_{d-2}^q+c$ has derivative $-D_{d-2}^q$, we see that degree $d$ primes in its factorization can occur with multiplicity at most one. 

   Let $p>2$. Assume that for prime $\wp$ of degree $d$ divides $L_{d-1}-c$. Then by above, $d\wp/dt$ is constant and thus, $\wp$ is a Wilson prime and so by  Theorem 5, we know that $\wp^{p-2}$ divides $L_{d-1}'$. Hence,  $\wp^{p-1}$ divides $L_{d-1}-c$.
   \end{proof}


\subsection{Remarks} (i) By Theorem 2 above,  these special Wilson primes can be also described as the prime basis $a$  of degree $d$ for which there are no Wieferich primes, and that these exist only if $p$ divides $d$, or $d=1$. 

(ii) It is clear from the above factorizations that the primes of degree less than $d$ do not divide these quantities.  As mentioned in the examples below, many very large degree (than d) primes can occur, and we do not know their characterization. Since the quantities $[i], L_i, D_i$ are invariant for translations $t\rightarrow t+c$, $c\in \F_q$,  the prime factorization has orbits under these.  This explains multiplicities or number of some large given degree primes which occur in the factorization. 

(iii) While the number of primes of degree $d$ is of the order $q^d/d$, this exponent $d$ becomes $2d/p$ and $d/p$ respectively (under naive randomness assumptions) for Wilson and special Wilson primes, when $p>2$. 


\subsection{Examples}  (0) If $d=1$, $L_{d-1}^{q-1}-1=0$, and all the degree 1 primes are of the required form having constant derivatives.  
In the next case,  $d=p$, it follows from \cite[Thm. 7.1]{binom} and we know even that all these are Artin-Schreier primes,  thus have derivatives $-1$ and already divide $L_{p-1}+1$. 
In higher degrees, we thus get generalizations of these primes and they can occur for any $c$ in general, but for some (low) degrees there are none for some or for all $c$'s. 
(e.g., for $q=2, d=8$ or $q=3, d=9, c=-1$ or $q=4, d=4, c=1$ there are none) 

(1) If $q=3, d=6$, there are total $116$ primes of degree $d$, out of which $6$ have constant derivative, and $15$ have vanishing second derivative. 
 Then the degree $363$ quantity $L_5+1$ ($L_5-1$ respectively) is the  product of the three degree $6$ primes with derivative $1$ (derivative $-1$ respectively) each with multiplicity $2$, three degree $14$ primes and three degree $95$  primes (all with multiplicity 1). The polynomial in the Theorem 5 has degree $360$ and is a product of the $15$ degree 
 $6$ primes with vanishing second derivative, three primes each of degrees $28, 24, 20$, two of degree  $18$, three of degree 2,  and the three primes of degree $1$ each with multiplicity $4$. If $q=3$, $d=9$, there are six primes of degree $9$ dividing $L_8^2-1$ they all divide $L_8-1$.  If $q=3, d=12$, there are no primes of degree $12$ with constant derivative. 

(2) If $q=2$, $d=14$, there are total $1161$ primes of degree $14$, out of which $12$ have constant derivative. 
The factorization of the degree $8192$ polynomial $L_{13}+1$ is the product of exactly the 12 conjectured primes  above, one prime each of degree $22$, $128$, and $9260$, and 
two primes each of degree $1156, 2246$. (Here I assume that the SAGE factor command indeed factored into primes.)

\subsection{Questions} Here are some of the natural questions that arise: 

(1) We proved  that for Wilson primes, the Wilson congruence holds modulo $\wp^{p-1}$. Does it ever 
(or infinitely often) hold modulo even higher power, if $d>1, p>2$? Often we have only proved lower bounds for the multiplicities, 
what are the exact multiplicities? 

(2) Are there nice generalizations of these phenomena for other function field situations, say even in class number one? 

(3) What are the distributions 
in the congruence classes when we do not have supercongruence? (i.e.,  when do not have the zero class modulo $\wp^2$.) 

(4) Interestingly,  the three derivatives appear in parallel fashion in the theorems, though the Fermat quotient is more like (negative) reciprocal of derivative of $\wp$.

{\bf Acknowledgments}: We thank the National University of Singapore and the Tata Institute of Fundamental Research. Preparing for the talks I gave  there,  in June, August 2022 respectively, on \cite{fw} results,  led to rethinking of these topics after several years and led to these new results (Theorems 4, 5,   and a conjecture which later became Theorem  7 respectively).  We used SAGE online cell server for numerical verifications.

\end{document}